\documentclass[11pt]{article}
\usepackage{a4wide}
\usepackage{amsmath,amssymb,amsthm}

\newcommand{\R}{\mathbb{R}}

\newcommand{\e}{\mathrm{e}}
\newcommand{\MM}{\mathcal{M}}
\newcommand{\PP}{\mathcal{P}}

\providecommand{\abs}[1]{\lvert #1\rvert}

\providecommand{\pos}{\mathop{\rm pos}\nolimits}

\providecommand{\tq}{\,\vert\,}
\providecommand{\vol}{\mathop\mathrm{vol}\nolimits}

\newtheorem{Thm}{Theorem}
\newtheorem{Lem}[Thm]{Lemma}
\newtheorem{Prop}[Thm]{Proposition}
\newtheorem{Cor}[Thm]{Corollary}

\theoremstyle{definition}
\newtheorem{Def}[Thm]{Definition}
\theoremstyle{remark}

\title{Partitions and functional Santal\'o inequalities}

\author{Joseph Lehec
\footnote{LAMA (UMR CNRS 8050) Universit\'e Paris-Est}}

\date{September 2008}

\begin{document}
\maketitle
\begin{abstract}
We give a direct proof of a functional Santal\'o inequality due to Fradelizi and Meyer. This provides a new proof of the Blaschke-Santal\'o inequality. The argument combines a logarithmic form of the Pr\'ekopa-Leindler inequality and a partition theorem of Yao and Yao.
\bigskip

\noindent
Published in Arch. Math. 92 (1) (2009) 89-94.
\end{abstract}
\section*{Introduction}
If $A$ is a subset of $\R^n$ we let $A^\circ$ be the polar of $A$:
\[ A^\circ = \{ x\in\R^n \tq \forall y\in A, \; x\cdot y \leq 1 \} ,\]
where $x\cdot y$ denotes the scalar product of $x$ and $y$. We denote the Euclidean norm of $x$ by $\abs{x} = \sqrt{x\cdot x}$. Let $K$ be a subset of $\R^n$ with finite measure. The Blaschke-Santal\'o inequality states that there exists a point $z$ in $\R^n$ such that
\begin{equation}
\label{classique-santalo}
 \vol_n (K)  \vol_n (K-z)^\circ \leq 
        \vol_n (B_2^n) \vol_n(B_2^n)^\circ = v_n^2 ,
\end{equation}
where $\vol_n$ stands for the Lebesgue measure on $\R^n$, $B_2^n$ for the Euclidean ball and $v_n$ for its volume. It was first proved by Blaschke in dimension $2$ and $3$ and Santal\'o \cite{santalo} extended the result to any dimension. We say that an element $z$ of $\R^n$ satisfying \eqref{classique-santalo} is a \emph{Santal\'o point} for $K$. \\
Throughout the paper a \emph{weight} is an measurable function $\rho: \R_+ \to \R_+$ such that for any $n$, the function $x\in \R^n\mapsto \rho(\abs{x})$ is integrable.
\begin{Def}
Let $f$ be a non-negative integrable function on $\R^n$, and $\rho$ be a weight. We say that $c\in \R^n$ is a Santal\'o point for $f$ with respect to $\rho$ if the following holds: for all non-negative Borel function $g$ on $\R^n$, if
\begin{equation}
\label{duality}
\forall x,y \in \R^n, \quad x\cdot y \geq 0 \; \Rightarrow \; 
f(c + x) g(y)  \leq  \rho \bigl( \sqrt{x\cdot y} \bigr) ^2 , 
\end{equation}
then
\begin{equation}
\label{conclusion}
\int_{\R^n} f(x) \, dx  \int_{\R^n} g(y) \, dy \leq  \bigl( \int_{\R^n}  \rho ( \abs{x} ) \, dx \bigr)^2 .
\end{equation}  
\end{Def}
Heuristically, the choice of the weight $\rho$ gives a notion of duality (or polarity) for non-negative functions. Our purpose is give a new proof of the following theorem, due to Fradelizi and Meyer~\cite{fradelizi-meyer}.
\begin{Thm}
\label{thmfm}
Let $f$ be non-negative and integrable. There exists $c\in\R^n$ such that $c$ is a Santal\'o point for $f$ with respect to any $\rho$. Moreover, if $f$ is even then $0$ is a Santal\'o point for $f$ with respect to any weight.
\end{Thm}
The even case goes back to Keith Ball in \cite{these-ball}, this was the first example of a functional version of \eqref{classique-santalo}. Later on, Artstein, Klartag and Milman~\cite{artstein-klartag-milman} proved that any integrable $f$ admits a Santal\'o point with respect to the weight $t\mapsto\e^{-t^2/2}$. Moreover in this case the barycenter of $f$ suits (see \cite{simpleFS}). Unfortunately this is not true in general; indeed, taking 
\begin{align*} 
f & = \boldsymbol{1}_{(-2,0)} + 4  \boldsymbol{1}_{(0,1)} \\ 
g & = \boldsymbol{1}_{(-0.5,0]} + \tfrac{1}{4} \boldsymbol{1}_{(0,1)} \\
\rho & = \boldsymbol{1}_{[0,1]} ,
\end{align*}
it is easy to check that $f$ has its barycenter at $0$, and that $f(s) g(t) \leq \rho \bigl( \sqrt{st} \bigr)^2 $ as soon as $st\geq 0$. However
\[ \int_\R f(s) \, ds \int_\R g(t) \, dt = \frac{9}{2} > 
4 = \bigl( \int_\R \rho{ (\abs{r}) } \, dr \bigr)^2. \]
To prove the existence of a Santal\'o point, the authors of \cite{fradelizi-meyer} use a fixed point theorem and the usual Santal\'o inequality (for convex bodies). Our proof is direct, in the sense that it does not use the Blashke Santal\'o inequality; it is based on a special form of the Pr\'ekopa-Leindler inequality and on a partition theorem due to Yao and Yao~\cite{yaoyao}.  \\
Lastly, the Blaschke-Santal\'o inequality follows very easily from Theorem~\ref{thmfm}: we let the reader check that if $c$ is a Santal\'o point for $\boldsymbol{1}_K$ with respect to the weight $\boldsymbol{1}_{[0,1]}$ then $c$ is a Santal\'o point for $K$.
\section{Yao-Yao partitions}
In the sequel we consider real affine spaces of finite dimension. If $E$ is such a space we denote by $\vec{E}$ the associated vector space. We say that $\PP$ is a partition of $E$ if $\cup \PP = E$ and if the interiors of two distinct elements of $\PP$ do not intersect. 
For instance, with this definition, the set $\{(-\infty,a],[a,+\infty)\}$ is a partition of $\R$. We define by induction on the dimension a class of partitions of an $n$-dimensional affine space.
\begin{Def}
\label{YYP}
If $E=\{c\}$ is an affine space of dimension $0$, the only possible partition $\PP=\{c\}$ is a Yao-Yao partition of $E$, and its center is defined to be $c$. \\
Let $E$ be an affine space of dimension $n\geq 1$. A set $\PP$ is said to be a Yao-Yao partition of $E$ if there exists an affine hyperplane $F$ of $E$, a vector $v\in \vec{E} \backslash \vec{F}$ and two Yao-Yao partitions $\PP_+$ and $\PP_-$ of $F$ \emph{having the same center $c$} such that 
\[ \PP= \bigl\{  A + \R_- v \tq A\in \PP_- \bigr\} \cup \bigl\{ A + \R_+ v \tq A\in \PP_+ \Bigr\} ,  \]
The center of $\PP$ is then $x$.
\end{Def}
If $A$ is a subset of $\vec{E}$ we denote by $\pos(A)$ the positive hull of $A$, that is to say the smallest convex cone containing $A$. \\
A Yao-Yao partition $\PP$ of an $n$-dimensional space $E$ has $2^n$ elements and for each $A$ in $\PP$ there exists a basis $v_1, \dotsc , v_n$ of $\vec{E}$ such that 
\begin{equation}
\label{pos} 
A = c + \pos(v_1, \dotsc , v_n) ,
\end{equation}
where $c$ is the center of $\PP$.
Indeed, assume that $\PP$ is defined by $F,v,\PP_+$ and $\PP_-$ (see Definition~\ref{YYP}). Let $A\in \PP_+$ and assume inductively that there is a basis $v_1,\dotsc, v_{n-1}$ of $\vec{F}$ such that $A = c  + \pos ( v_1,\dotsc, v_{n-1} )$. Then $A + \R_+ v =  c + \pos( v , v_1, \dotsc, v_{n-1} )$. 

A fundamental property of this class of partitions is the following
\begin{Prop}
\label{halfspace}
Let $\PP$ be a Yao-Yao partition of $E$ and $c$ its center. Let $\ell$ be an affine form on $E$ such that $\ell(c)=0$. Then there exists $A\in \PP$ such that $\ell(x) \geq 0$ for all $x\in A$. Moreover there is at most one element $A$ of $\PP$ such that $\ell(x) > 0$ for all $x\in A\backslash\{c\}$.
\end{Prop}
\begin{proof}
By induction on the dimension $n$ of $E$. When $n=0$ it is obvious, we assume that $n\geq 1$ and that the result holds for all affine spaces of dimension $n-1$. Let $\ell$ be an affine form on $E$ such that $\ell(c)= 0$. We introduce $F,v,\PP_+,\PP_-$ given by Definition~\ref{YYP}.  By the induction assumption, there exists $A_+ \in \PP_+$ and $A_-\in\PP_-$ such that 
\[ \forall y\in A_+ \cup A_- \quad  \ell(y) \geq 0 . \]
If $\ell(c+v) \geq 0$ then $\ell(x+tv)\geq 0$ for all $x\in A_+$ and $t\in\R_+$, thus $\ell(x) \geq 0$ for all $x\in A_+ +\R_+v$. If on the contrary $\ell(c+v)\leq 0$ then $\ell(x)\geq 0$ for all $x\in A_- + \R_- v$, which proves the first part of the proposition. The proof of the `moreover' part is similar.
\end{proof}
The latter proposition yields the following corollary, which deals with dual cones: if $C$ is cone of $\R^n$ the dual cone of $C$ is by definition
\[ C^* = \{ y \in \R^n \tq \forall x\in C, \; x\cdot y \geq 0 \} . \]
\begin{Cor}
\label{polarity}
Let $\PP$ be a Yao-Yao partition of $\R^n$ centered at $0$. Then
\[ \PP^* := \{ A^* \tq A \in \PP \} \]
is also a partition of $\R^n$.
\end{Cor}
Actually the dual partition is also a Yao-Yao partition centered at $0$ but we will not use this fact. 
\begin{proof}
Let $x\in\R^n$ and $\ell:y\in\R^n\mapsto x\cdot y$. By the previous proposition there exists $A\in \PP$ such that $\ell(y)\geq 0$ for all $y\in A$. Then $x \in A^*$. Thus $\cup \PP^* = \R^n$. Moreover if $x$ belongs to the interior of $A^*$, then for all $y\in A\backslash \{0\}$ we have $\ell(y)>0$. Again by the proposition above there is at most one such $A$. Thus the interiors of two distinct elements of $\PP^*$ do not intersect.
\end{proof}
We now let $\MM(E)$ be the set of Borel measure $\mu$ on $E$ which are finite and which satisfy $\mu(F)=0$ for any affine hyperplane $F$. 
\begin{Def}
\label{equipartition}
Let $\mu\in\MM(E)$, a Yao-Yao equipartition $\PP$ for $\mu$ is a Yao-Yao partition of $E$ satisfying 
\begin{equation}
\label{equip}
\forall A\in\PP, \quad  \mu(A) = 2^{-n} \mu(E).
\end{equation}
We say that $c\in E$ is a Yao-Yao center of $\mu$ if $c$ is the center of a Yao-Yao equipartition for $\mu$. 
\end{Def}
Here is the main result concerning those partitions.
\begin{Thm}
\label{main}
Let $\mu \in \MM( \R^n )$, there exists a Yao-Yao equipartition for $\mu$. Moreover, if $\mu$ is even then $0$ is a Yao-Yao center for $\mu$.
\end{Thm}
It is due to Yao and Yao \cite{yaoyao}. They have some extra hypothesis on the measure and their paper is very sketchy, so we refer to \cite{monyaoyao} for a proof of this very statement.
\section{Proof of the Fradelizi-Meyer inequality}
In this section, all integrals are taken with respect to the Lebesgue measure. Let us recall the Pr\'ekopa-Leindler inequality, which is a functional form of the famous Brunn-Minkowski inequality, see for instance \cite{ball} for a proof and selected applications. If $\varphi_1,\varphi_2,\varphi_3$ are non-negative and integrable functions on $\R^n$ satisfying $\varphi_1(x)^\lambda \varphi_2(y)^{1-\lambda} \leq \varphi_3(\lambda x+ (1-\lambda) y)$ for all $x,y$ in $\R^n$ and for some fixed $\lambda \in (0,1)$, then
\[\Bigl( \int_{\R^n}  \varphi_1  \Bigr)^\lambda  \Bigl( \int_{\R^n}  \varphi_2 \Bigr)^{1-\lambda}  \leq  \int_{\R^n}  \varphi_3. \]
The following lemma is a useful (see \cite{fradelizi-meyer,these-ball}) logarithmic version of Pr\'ekopa-Leindler. We recall the proof for completeness.
\begin{Lem}
\label{gm}
Let $f_1,f_2,f_3$ be non-negative Borel functions on $\R_+^n$ satisfying
\[ f_1(x) f_2(y) \leq \Bigl( f(\sqrt{x_1y_1},\dotsc,\sqrt{x_ny_n}) \Bigr)^2 . \]
for all $x,y$ in $\R_+^n$. Then
\begin{equation}
\label{pll}
\int_{\R_+^n} f_1 \int_{\R_+^n} f_2 \leq \Bigl( \int_{\R_+^n} f_3 \Bigr)^2 .
\end{equation}
\end{Lem}
\begin{proof}
For $i=1,2,3$ we let
\[ g_i(x) = f_i( \e^{x_1}, \dotsc , \e^{x_n} ) \e^{x_1+\dotsb+x_n} . \]
Then by change of variable we have
\[ \int_{\R^n} g_i = \int_{\R_+^n} f_i . \]
On the other hand the hypothesis on $f_1,f_2,f_3$ yields
\[  g_1(x)g_2(y) \leq g_3(\tfrac{x+y}{2}) ,\]
for all $x,y$ in $\R^n$. Then by Pr\'ekopa-Leindler
\[ \int_{\R^n} g_1 \int_{\R^n} g_2 \leq \Bigl( \int_{\R^n} g_3 \Bigr)^2 . \qedhere \]
\end{proof}
\begin{Thm}
Let $f$ be a non-negative Borel integrable function on $\R^n$, and let $c$ be a Yao-Yao center for the measure with density $f$. Then $c$ is a Santal\'o point for $f$ with respect to any weight.
\end{Thm}
Combining this result with Theorem~\ref{main} we obtain a complete proof of the Fradelizi-Meyer inequality.
\begin{proof}
It is enough to prove that if $0$ is a Yao-Yao center for $f$ then $0$ is a Santal\'o point. Indeed, if $c$ is a center for $f$ then $0$ is a center for 
\[ f_c:x \mapsto f(c+x) . \]
And if $0$ is a Santal\'o point for $f_c$ then $c$ is a Santal\'o point for $f$. \\
Let $\PP$ be a Yao-Yao equipartition for $f$ with center $0$. Let $g$ and $\rho$ be such that \eqref{duality} holds (with $c=0$). Let $A \in \PP$, by \eqref{pos}, there exists an operator $T$ on $\R^n$ with determinant $1$ such that $A= T \bigl(  \R_+^n \bigr)$. Let $S=(T^{-1})^*$, then $S\bigl( \R_+^n \bigr) = A^*$. Let $f_1 = f \circ T$, $f_2 = g \circ S$ and $f_3(x) = \rho(\abs{x})$. Since for all $x,y$ we have $T(x)\cdot S(y) = x\cdot y$, we get from \eqref{duality}
\[ f_1(x) f_2(y) \leq \rho( \sqrt{ x \cdot y } )^2 = f_3( \sqrt{x_1 y_1} , \dotsc , \sqrt{x_n y_n} )^2 , \]
for all $x,y$ in $\R_+^n$. Applying the previous lemma we get \eqref{pll}. By change of variable it yields
\[ \int_{A} f  \int_{ A^* } g  \leq \Bigl( \int_{\R_+^n} \rho( \abs{x} ) \, dx\Bigr)^2 . \]
Therefore 
\begin{equation}
\label{sum}
\sum_{A\in \PP} \int_A f \int_{A*} g \leq 2^n \Bigl( \int_{\R_+^n} \rho( \abs{x} ) \, dx \Bigr)^2 .
\end{equation} 
Since $\PP$ is an equipartition for $f$ we have for all $A\in\PP$
\[ \int_A f  =  2^{-n} \int_{\R^n} f . \]
By Corollary~\ref{polarity}, the set $\{ A^*, A\in \PP\}$ is a partition of $\R^n$, thus
\[ \sum_{A\in \PP} \int_{A^*} g  = \int_{\R^n} g  . \]
Inequality~\eqref{sum} becomes
 \[ \int_{\R^n} f \int_{\R^n} g  \leq 4^n \Bigl( \int_{\R_+^n} \rho(\abs{x}) \, dx \Bigr)^2  , \]
and of course the latter is equal to $\bigl( \int_{\R^n} \rho(\abs{x}) \, dx \bigl)^2$. 
\end{proof}
\end{document}